\documentclass{amsart}

\usepackage{amsfonts}
\usepackage{amsmath}
\usepackage{amssymb}
\usepackage{amsthm}
\usepackage[all]{xy}
\usepackage{verbatim}

\newtheorem{thrm}{Theorem}
\newtheorem{lemma}[thrm]{Lemma}
\newtheorem{prop}[thrm]{Proposition}
\newtheorem{cor}[thrm]{Corollary}

\theoremstyle{definition}
\newtheorem*{defn}{Definition}
\newtheorem{ex}{Example}

\DeclareMathOperator{\coker}{coker}
\DeclareMathOperator{\Hom}{Hom}

\DeclareMathOperator{\End}{End}
\DeclareMathOperator{\im}{im}

\DeclareMathOperator{\md}{mod}

\DeclareMathOperator{\Sub}{Sub}

\begin{document}

\newcommand{\N}{\mathbb{N}}
\newcommand{\M}{\mathcal{M}}
\newcommand{\Z}{\mathbb{Z}}
\newcommand{\R}{\mathbb{R}}
\newcommand{\C}{\mathbb{C}}
\newcommand{\Q}{\mathbb{Q}}
\newcommand{\matr}[1]{\left(\begin{matrix}#1\end{matrix}\right)}
\newcommand{\smatr}[1]{\left(\begin{smallmatrix}#1\end{smallmatrix}\right)}

\renewcommand{\r}{\mathfrak{r}}
\renewcommand{\P}{\mathcal{P}}
\renewcommand{\le}[1]{\leq_{\mathrm{#1}}}

\title{Module Degenerations and Finite Field Extensions}
\author{Nils M. Nornes}

\begin{abstract}
Degeneration of modules is usually defined geometrically, but due to results of Zwara and Riedtmann we can also define it in purely homological terms. This homological definition also works over fields that are not algebraically closed. Let $k$ be a field, $K$ a finite extension of $k$ and $\Lambda$ a $k$-algebra. Then any $K\otimes_k\Lambda$-module is also a $\Lambda$-module. We study how the isomorphism classes, degeneration and hom-order differ depending on whether we work over $\Lambda$ or $K\otimes_k\Lambda$.
\end{abstract}
\maketitle
\section{Introduction}
Let $k$ be a field, $K$ a normal finite field extension of $k$ and  $Q$  a quiver. Since $K$-vector spaces are also $k$-vector spaces and all $K$-linear maps are $k$-linear, any $K$-representation of $Q$ is also a $k$-representation. But, since not all $k$-linear maps are $K$-linear, two nonisomorphic $K$-representations may be isomorphic as $k$-representations.

 \begin{ex}\label{Kron1}
 Let $Q$ be the Kronecker quiver and consider the $\C Q$-modules 
 $$M: \xymatrix{\C\ar@<1ex>[r]^1\ar[r]_i&\C},\qquad N:\xymatrix{\C\ar@<1ex>[r]^1\ar[r]_{-i}&\C}.$$
 $M$ and $N$ are not isomorphic as $\C Q$-modules, but if we view them as $\R Q$-modules, there is an isomorphism given by complex conjugation.
 \end{ex}

More generally, if $\Lambda$ is a $k$-algebra, then two $K\otimes_k\Lambda$-modules may be isomorphic in $\md\Lambda$, the category of finite-dimensional $\Lambda$-modules, but nonisomorphic in $\md K\otimes_k\Lambda$.

When we need to specify which algebra two modules are isomorphic over, we will add a superscript to the isomorphism sign, e.g. $M\simeq^{\R Q} N$.

The $\Lambda$-isomorphism class of a given $K\otimes_k\Lambda$-module splits into a number of $K\otimes_k\Lambda$-isomorphism classes.  
In section \ref{isosection} we give a complete description of these isomorphism classes.

 Since isomorphism classes depend on which algebra we are working over, so do the degeneration order and the $\Hom$-order. 

Degeneration of modules is usually defined geometrically. For a natural number $d$ and a $k$-algebra $\Lambda$, let $\md_d\Lambda$ be the set of algebra homomorphisms from $\Lambda$ to $\M_d(k)$, the ring of $d\times d$-matrices with entries in $k$. Given a homomorphism $\mu\in\md\Lambda$ we can make a module structure on $k^d$. For any $\lambda\in\Lambda,x\in k^d$, we define  $\lambda x:=\mu(\lambda)\cdot x$, where $x$ is viewed as a column vector and the multiplication on the right hand side is just matrix multiplication. This lets us identify $\md_d\Lambda$ with the set of $\Lambda$-module structures on $k^d$. The set $\md_d\Lambda$ is actually an affine variety, and we say that a module $M$ degenerates to a module $N$ if $N$ is in the closure of the isomorphism class of $M$. 

This definition only works when $k$ is algebraically closed, and in this paper we want to look at other fields. In \cite{Zwara}, G. Zwara showed that there is an equivalent module theoretic way to describe degeneration, and we will use this description as the definition.

\begin{defn}
Let $M$ and $N$ be modules in $\md \Lambda$. $M$ degenerates to $N$ if there exists a module $X\in\md\Lambda$ and an exact sequence
$$\xymatrix{0\ar[r]&X\ar[r]&X\oplus M\ar[r]&N\ar[r]&0}.$$
We denote this by $M\le{deg} N$. An exact sequence of the above form is called a Riedtmann sequence.
\end{defn}

This definition works for any field. With this definition it is not obvious that $\le{deg}$ is a partial order, but this was shown by G. Zwara in \cite{Zwara2}.

The degeneration order does not behave nicely with respect to cancellation of common direct summands, so in \cite{Riedtmann} C. Riedtmann introduced another order.

\begin{defn}
Let $M$ and $N$ be $\Lambda$-modules. $M$ virtually degenerates to $N$ if there exists $Z\in\md\Lambda$ such that $M\oplus Z\le{deg}N\oplus Z$. We denote this by $M\le{vdeg} N$. 
\end{defn}

$M\le{deg}N$ clearly implies $M\le{vdeg}N$, but for some algebras the virtual degeneration is strictly finer. This was first shown by an example due to J. Carlson (see \cite{Riedtmann}). 

The last partial order we want to study in this paper is the Hom-order, which is based on the dimensions of Hom-spaces. We will denote the $k$-dimension of $\Hom_\Lambda(M,N)$ by ${}_\Lambda[M,N]$.

\begin{defn}
Given two $\Lambda$-modules $M$ and $N$, $M\le{Hom} N$ if ${}_\Lambda [X,M]\leq {}_\Lambda[X,N]$ for all $X\in\md\Lambda$ (or, equivalently, if ${}_\Lambda [M,X]\leq {}_\Lambda[N,X]$ for all $X\in\md\Lambda$).
\end{defn}

The relation  $\le{\Hom}$ is clearly reflexive and transitive. In \cite{Aus}, M. Auslander showed that if $M\not\simeq N$ then there exists an $X\in\md\Lambda$ such that ${}_\Lambda[X,M]\neq{}_\Lambda[X,N]$, which shows that $\le{\Hom}$ is also antisymmetric. 

$M\le{vdeg}N$  implies $M\le{Hom}N$, but it is not known if $\le{Hom}$ is strictly finer.

However, if the algebra is representation-finite, all three  orders are the same. This was shown for algebras over algebraically closed fields by G. Zwara in \cite{zwararepfin} and generalized to arbitrary artin algebras by S. O. Smal\o{} in \cite{sverremilano}.

As with isomorphisms, we add a superscript when we need to specify which algebra we are considering.

In section \ref{posection} we give several examples where $\le{deg}^\Lambda$ differs from $\le{deg}^{K\otimes_k\Lambda}$. We also give some examples of modules $M,N$ where $M\oplus M$ degenerates to $N\oplus N$ but $M$ does not degenerate to $N$. For some algebras $\Lambda$ the $K\otimes_k\Lambda$-isomorphism classes are the same as the $\Lambda$-isomorphism classes. We show that in these cases $\le{Hom}^{K\otimes_k\Lambda}$ and $\le{Hom}^\Lambda$ are also the same.

In section \ref{ringsection} we show that if the endomorphism ring of a module is a division ring, then the module is minimal in the degeneration- and Hom-orders.

For background on representation theory of algebras we refer the reader to \cite{ARS}. For an introduction to degenerations of modules, see \cite{sverremilano}.

\section{Isomorphism classes}\label{isosection}

Let $k$ be a field, $K$ a separable finite extension of $k$ and $\Lambda$ a $k$-algebra. Let $\Gamma=K\otimes_k\Lambda$. For any $\Lambda$-module $M$ we give $K\otimes_k M$ a $\Gamma$-module structure by $(x\otimes\lambda)\cdot(y\otimes m)=xy\otimes\lambda m$.  Since $\Lambda$ is a subring of $\Gamma$ any $\Gamma$-module is also a $\Lambda$-module.

Furthermore, any $\Gamma$-homomorphism is a $\Lambda$-homomorphism, so $X\simeq^\Gamma Y$ implies $X\simeq^\Lambda Y$. But, as Example \ref{Kron1} shows, the reverse implication does not hold.

In Example \ref{Kron1} we see that the $\R Q$-isomorphism class of $M$ contains two $\C Q$-isomorphism classes, and one is in some sense a complex conjugate of the other. On the other hand,  the $\R Q$-isomorphism class of the module
$$X_a: \xymatrix{\C\ar@<1ex>[r]^1\ar[r]_a&\C}$$ contains only one $\C Q$-isomorphism class if $a\in\R$.  If $a$ is not real it has two $\C Q$-isomorphism classes. Note also that when $a$ is real there exists a $\R Q$-module $Y_a$ such that $X_a\simeq^{\C Q}\C\otimes_\R Y_a$, whereas when $a$ is not real there is no such $\R Q$-module.
 
Similarly,  for any indecomposable $\C Q$-module $M$ its $\R Q$-isomorphism class contains either one or two $\C Q$-isomorphism classes. When there are two, a module in the second class can be constructed from $M$ by complex conjugation. 
 
 More generally, if $K$ is a normal extension of $k$ of degree $n$, the $\Lambda$-isomorphism class of an indecomposable $\Gamma$-module splits into at most $n$ $\Gamma$-isomorphism classes, and they are related by $k$-automorphisms of $K$.

Given a $k$-automorphism of $K$ and a $\Gamma$-module $M$, we can construct a $\Gamma$-module that is $\Lambda$-isomorphic to $M$ in the following way.

 Let $\phi$ be a $k$-automorphism on $K$, and let $M$ be a $\Gamma$-module. We construct a new $\Gamma$-module $M^\phi$ by setting $M^\phi=M$ as $k$-spaces, and letting the multiplication be given by $(x\otimes\lambda)\cdot_{M^\phi} m=(\phi(x)\otimes \lambda)\cdot_M m$. Now the identity on $M$ gives us a $\Lambda$-isomorphism $\hat\phi:M^\phi\to M$, where for any $x\in K$ and $m\in M$ we have $\hat{\phi}(xm)=\phi(x)\hat{\phi}(m)$.
 
When $K$ is a normal extension of $k$, let $G(K/k)$ denote its Galois group.
 
We are now ready to prove the main result of this section.

\begin{thrm}
Let $M\in\md\Gamma$. The multiplication map $$\mu_M:K\otimes_kM\to M$$$$x\otimes m\mapsto xm$$ is a split epimorphism of $\Gamma$-modules.

Furthermore, if $K$ is a normal extension of $k$, then we have $$K\otimes_k M\simeq^\Gamma \bigoplus_{\phi\in G(K/k)}M^\phi.$$
\end{thrm}

\begin{proof}
We prove the first part by constructing a splitting of $\mu_M$. 

 Let $\mu:K\otimes_kK\to K$ be the map given by $\mu(x\otimes y)=xy$. This is a $K\otimes_kK$-module epimorphism. Since $K$ is separable we have by Lemma 9.2.8 and Theorem 9.2.11 in \cite{Weibel} that $K$ is a projective $K\otimes_kK$-module, and thus $\mu$ splits. Let $\nu:K\to K\otimes_kK$ be a splitting of $\mu$.

We first consider $M=\Gamma=K\otimes_k\Lambda$. 
 $K\otimes_k\Gamma=K\otimes_kK\otimes_k\Lambda$ is a $\Gamma$-module with multiplication $(x\otimes\lambda)\cdot (y\otimes z\otimes \kappa)=xy\otimes z\otimes \lambda\kappa$. 
Let $\nu_\Gamma:K\otimes_k\Lambda\to K\otimes_kK\otimes_k\Lambda$ be given by $\nu_\Gamma(x\otimes\lambda)=\nu(x)\otimes\lambda$. This is a $\Gamma$-module homomorphism and a splitting of $\mu_\Gamma$. We now show that for any $f\in\Hom_\Gamma(\Gamma,\Gamma)$ the following diagram commutes.
$$\xymatrix{\Gamma\ar[r]^f\ar[d]^{\nu_\Gamma}&\Gamma\ar[d]^{\nu_\Gamma}\\K\otimes_k\Gamma\ar[r]^{K\otimes_k f}&K\otimes_k\Gamma}$$
The homomorphism $f$ is given by right multiplication with an element in $\Gamma$, and it is enough to check that the diagram commutes for all generators of $\Gamma$. Let $x,y\in K$, $\alpha,\beta\in\Lambda$ and $f=-\cdot y\otimes \beta$. Then $\nu_\Gamma f(x\otimes\alpha)=\nu_\Gamma(xy\otimes\alpha\beta)=\nu(xy)\otimes\alpha\beta$ and $K\otimes_kf\nu_\Gamma(x\otimes\alpha)=K\otimes_kf(\nu(x)\otimes\alpha)=\nu(x)\otimes\alpha\cdot1\otimes y\otimes\beta$, and since $\nu$ is a $K\otimes_k K$-homomorphism we have $ \nu(x)\otimes\alpha\cdot1\otimes y\otimes\beta=\nu(xy)\otimes\alpha\beta$, so the diagram commutes.

For a free $\Gamma$-module $\Gamma^n$ let $\nu_{\Gamma^n}$ be given by $\nu_{\Gamma^n}((\gamma_1,\ldots,\gamma_n))=(\nu_\Gamma(\gamma_1),\ldots,\nu_\Gamma(\gamma_n))$. Then for any $f\in\Hom_\Gamma(\Gamma^a,\Gamma^b)$ the following diagram commutes.
$$\xymatrix{\Gamma^a\ar[r]^f\ar[d]^{\nu_{\Gamma^a}}&\Gamma\ar[d]^{\nu_{\Gamma^b}}\\K\otimes_k\Gamma^a\ar[r]^{K\otimes_k f}&K\otimes_k\Gamma^b}$$

For an arbitrary $M$, let $$\xymatrix{\Gamma^a\ar[r]^f&\Gamma^b\ar[r]^g&M}$$ be a free presentation. Then we construct $\nu_M$ from the commutative diagram
$$\xymatrix{\Gamma^a \ar[r]^f\ar[d]^{\nu_{\Gamma^a}}&\Gamma^b\ar[r]^g\ar[d]^{\nu_{\Gamma^b}}&M \ar@{.>}[d]^{\nu_M}\\K\otimes_k\Gamma^a\ar[r]^{K\otimes_kf}&K\otimes_k\Gamma^b\ar[r]^{K\otimes_kg}&K\otimes_kM}.$$
Since 
$$\xymatrix{\Gamma^b\ar[r]^g&M\\K\otimes_k\Gamma^b\ar[r]^{K\otimes_kg}\ar[u]^{\mu_{\Gamma^b}}&M\ar[u]^{\mu_M}}$$
also commutes and $\nu_\Gamma$ is a splitting of $\mu_\Gamma$, $\nu_M$  is a splitting of $\mu_M$.

Let $\phi\in G(K/k)$. Then we have a $\Gamma$-isomorphism $1\otimes\hat\phi:K\otimes M^\phi\to K\otimes M$, so $M^\phi$ is a summand of $K\otimes_kM$ and the composition $\iota_\phi:=1\otimes\hat\phi\circ \nu_{M^\phi}$ is the inclusion. Let $\theta\neq\phi$ be another element in $G(K/k)$, and let $m\in \im\iota_\phi\cap \im \iota_\theta$. Now we view $K\otimes_k M$ as a $K\otimes_k K\otimes_k \Lambda$-module. Since $m$ is in $\im\iota_\phi$, we have for any $x\in K$ that $x\otimes 1\otimes 1 \cdot m=1\otimes \phi(x)\otimes 1 \cdot m$. Thus we get $$1\otimes(\phi(x)-\theta(x))\otimes 1\cdot m=0$$ for all $x\in K$, which means that $m=0$. Thus $\im\iota_\phi\cap \im\iota_\theta=(0)$, so $M^\phi$ and $M^\theta$ are distinct summands. If $K$ is normal it follows that $$K\otimes_kM\simeq\bigoplus_{\phi\in G(K/k)}M^\phi.$$

\end{proof}

When $K$ is normal, this gives us a complete description of the $\Gamma$-modules that are $\Lambda$-isomorphic to a given $\Gamma$-module. 

\begin{cor}
Let $K$ be a normal extension of $k$, and let $M_1,\ldots,M_r$ be indecomposable $K\otimes_k\Lambda$-modules. If $M\simeq^\Lambda M_1\oplus\ldots\oplus M_r$, then there exist $\phi_1,\ldots,\phi_r\in G(K/k)$ such that $M\simeq^{K\otimes_k\Lambda} M_1^{\phi_1}\oplus\ldots\oplus M_r^{\phi_r}$.
\end{cor}

When $K$ is not normal, this does not hold. Then $\Lambda$-isomorphisms do not even preserve the number of indecomposable $\Gamma$-summands.
 
 \begin{ex} 
 Let $K=\Q(\alpha)$ where $\alpha$ is a root of $X^3-2$. $K$ is not a normal extension of $\Q$, and it has no nontrivial $\Q$-automorphisms. $K\otimes_\Q K$ as a module over itself decomposes to $K\oplus L$, and $L\simeq K^2$ as $K$-modules, but not as $K\otimes_\Q K$-modules. In fact $L$ is an indecomposable $K\otimes_\Q K$-module.
 
 \end{ex}
 
\section{Partial orders}\label{posection}

Given two $\Gamma$-modules $M$ and $N$, we can ask if $M$ degenerates to $N$ as a $\Gamma$-module, but also if $M$ degenerates to $N$  as a $\Lambda$-module.
 
 If we have $M\le{deg}^\Gamma N$, then there is an exact sequence of $\Gamma$-modules 
 $$\xymatrix{0\ar[r]&X\ar[r]&X\oplus M\ar[r]&N\ar[r]&0}.$$
 This is also an exact sequence of $\Lambda$-modules, so we also have $M\le{deg}^\Lambda N$. 
 
We have already seen examples where $M\simeq^\Lambda N$ but $M\not\simeq^\Gamma N$. These examples also show that $\Lambda$-degeneration does not imply $\Gamma$-degeneration. There are also proper $\Lambda$-degenerations that are not $\Gamma$-degenerations.

\begin{ex}
Consider the algebra $$\Lambda=\matr{\C&\C\\0&\R}\subseteq \M_2(\C).$$ This is a hereditary $\R$-algebra corresponding to the Dynkin graph $B_2$. We have that $\C\otimes_\R\Lambda\simeq\C Q$ as $\C$-algebras, where $Q$ is the quiver
$$Q:\xymatrix{ 1 & 2\ar[l]_\alpha \ar[r]^\beta & 3},$$
via the isomorphism $f: \C\otimes_\R\Lambda\to \C Q$ given by $f(1\otimes\smatr{1&0\\0&0})=(e_1+e_3)$, $f(1\otimes\smatr{i&0\\0&0})=i(e_1-e_3)$, $f(1\otimes\smatr{0&0\\0&1})=e_2$, $f(1\otimes\smatr{0&1\\0&0})=(\alpha+\beta)$ and $f(1\otimes\smatr{0&i\\0&0})=i(\alpha-\beta)$.

The simple $\C Q$-modules $S_1$ and $S_3$ are isomorphic as $\Lambda$-modules. 

The $\C Q$-modules
$$I_1:\xymatrix{\C&\C\ar[l]_1\ar[r]&0},\qquad I_3:\xymatrix{0&\C\ar[l]\ar[r]^1&\C}$$
are also isomorphic as $\Lambda$-modules. $I_1$ degenerates to $S_2\oplus S_3$ as a $\Lambda$-module, but not as a $\C Q$-module, and the same holds for $I_3\le{deg}^\Lambda S_2\oplus S_1$.
\end{ex}

In the above example all $\Lambda$-degenerations of $\Gamma$-modules can be decomposed into  $\Gamma$-degenerations and   $\Lambda$-isomorphisms. That is, for any modules $M,N\in\md\Gamma$ such that $M\le{deg}^\Lambda N$, there exist $M',N'\in\md\Gamma$ such that $M\simeq^\Lambda M'\le{deg}^\Gamma N'\simeq^\Lambda N$. This does not hold for all algebras.

\begin{ex}
Let $Q$ be the quiver $$Q:\xymatrix{\bullet\ar@<1ex>[r]\ar[r]&\bullet\ar@<1ex>[r]\ar[r]&\bullet}$$ and let $\Lambda=\R Q$ and $\Gamma= \C Q$. Consider the modules given by the following representations:
$$A:\xymatrix{0\ar@<1ex>[r]\ar[r]&\C\ar@<1ex>[r]^{\smatr{1\\0}}\ar[r]_{\smatr{0\\1}}&\C^2},\quad B:\xymatrix{\C^3\ar@<1ex>[r]^{\smatr{1&0&0\\0&1&0}}\ar[r]_{\smatr{0&1&0\\0&0&1}}&\C^2\ar@<1ex>[r]^{\smatr{1&0\\0&1}}\ar[r]_{\smatr{i&0\\0&i}}&\C^2}$$
$$C:\xymatrix{\C^3\ar@<1ex>[r]^{\smatr{1&0&0}}\ar[r]_{\smatr{0&1&0}}&\C\ar@<1ex>[r]\ar[r]&0},\quad X:\xymatrix{0\ar@<1ex>[r]\ar[r]&\C\ar@<1ex>[r]^{\smatr{1}}\ar[r]_{\smatr{i}}&\C}$$
Now there is an exact sequence $\xymatrix{0\ar[r]&A\ar[r]&B\ar[r]&C\ar[r]&0}$ in $\md\Lambda$, so we have $B\le{deg}^\Lambda A\oplus C$. However, we have ${}_\Gamma[X,A\oplus C]=1<{}_\Gamma[X,B]=2$, so $B\not\le{deg}^\Gamma A\oplus C$. Letting $\phi$ denote complex conjugation we also have $A^\phi\simeq^\Gamma A$, $C^\phi\simeq^\Gamma C$ and ${}_\Gamma[X,B^\phi]=2$, so there are no modules $M$ and $N$ such that $M\simeq^\Lambda B\le{deg}^\Gamma A\oplus C\simeq^\Lambda N$.
\end{ex}

$\Lambda$-isomorphisms do not preserve $\Gamma$-degenerations, and two $\Lambda$-isomorphic modules can behave quite differently in the $\Gamma$-degeneration order. For example, minimality is not preserved.
\begin{ex}
Let $Q$ be the Kronecker quiver and let $\Lambda=\R Q$ and $\Gamma=\C\otimes_\R\Lambda$. Let $M$, $N$ and $N'$ be the modules given by 
$$M:\xymatrix{\C^2\ar@<1ex>[r]^{\smatr{1&0\\0&1}}\ar[r]_{\smatr{i&0\\1&i}}&\C^2}$$
$$N:\xymatrix{\C^2\ar@<1ex>[r]^{\smatr{1&0\\0&1}}\ar[r]_{\smatr{i&0\\0&i}}&\C^2}.$$
$$N':\xymatrix{\C^2\ar@<1ex>[r]^{\smatr{1&0\\0&1}}\ar[r]_{\smatr{i&0\\0&-i}}&\C^2}$$

We have that $M\le{deg}^\Gamma N$, but $N'$ is minimal in the degeneration order of $\md \Gamma$  and $N\simeq^\Lambda N'$.
\end{ex}

For some algebras, e.g. $\Lambda=kQ$ where $Q$ is a simply laced Dynkin quiver, the isomorphism classes in $\md\Lambda$ and $\md K\otimes_k\Lambda$ are the same. It seems likely that in these cases the degeneration order should also be the same. The Hom-order is indeed the same.

\begin{thrm}\label{iso-hom}
Let $\Lambda$ be a $k$-algebra and $\Gamma=K\otimes_k\Lambda$. The following are equivalent:
\begin{enumerate}
\item \label{IsoDeg1} $M\simeq^\Lambda N\iff M\simeq^\Gamma N$ for all $M,N\in\md\Gamma$.
\item \label{IsoDeg2} $M\le{Hom}^\Lambda N\iff M\le{Hom}^\Gamma N$ for all $M,N\in\md\Gamma$.
\end{enumerate}
\end{thrm}

\begin{proof}
We always have that $M\simeq^\Gamma N\implies M\simeq^\Lambda N$. If $M\simeq^\Lambda N$, then we have $M\le{Hom}^\Lambda N$ and $N\le{Hom}^\Lambda M$. Assuming that \ref{IsoDeg2} holds, we then have $M\le{Hom}^\Gamma N$ and $N\le{Hom}^\Gamma M$, and thus $M\simeq^\Gamma N$. This shows that \ref{IsoDeg2} implies \ref{IsoDeg1}.
 
 Now assume that \ref{IsoDeg1} holds.

For any $k$-algebra $R$ and $R$-modules $A$ and $B$ we have $\Hom_{K\otimes_k R}(K\otimes_k A,K\otimes_k B)\simeq_K K\otimes_k \Hom_R(A,B)$. Thus for any $\Gamma$-modules $X$ and $M$ we have ${}_\Gamma[K\otimes_kX,K\otimes_kM]=n\cdot{}_\Lambda[X,M]$, where $n$ is the degree of $K$. But given \ref{IsoDeg1} we also have ${}_\Gamma[K\otimes_kX,K\otimes_kM]={}_\Gamma[X^n,M^n]=n^2\cdot{}_\Gamma[X,M]$, so we get $n\cdot{}_\Gamma[X,M]={}_\Lambda[X,M]$. It follows that $M\le{Hom}^\Lambda N$ implies $M\le{Hom}^\Gamma N$. 

Assume that $M\le{Hom}^\Gamma N$. For every $\Lambda$-module $X$ we have $n\cdot {}_\Lambda[X,M]={}_\Gamma[K\otimes_k X, M]\leq {}_\Gamma[K\otimes_k X,N]=n\cdot {}_\Lambda[X,M]$, and thus $M\le{Hom}^\Lambda N$.

\end{proof}

This leaves the question of whether the same result holds for degenerations and virtual degenerations. If $\Lambda$ satisfies the statements of Theorem \ref{iso-hom}, do we also have that $\le{deg}^\Lambda$ and $\le{deg}^\Gamma$ are the same?

For  representation-finite algebras all three orders are the same, so in that case the answer is yes. It looks like all algebras that satisfy the statements of Theorem \ref{iso-hom} are representation-finite, so one option is to try to prove that.

Another possible way to prove Theorem \ref{iso-hom} for degenerations is to use a Riedtmann sequence in $\md\Lambda$ to construct a Riedtmann sequence in $\md\Gamma$.

When isomorphism classes are the same, we have that for any short exact sequence 
$$\xymatrix{0\ar[r]&A\ar[r]&B\ar[r]&C\ar[r]&0}$$ in $\md\Lambda$ there is a short exact sequence
$$\xymatrix{0\ar[r]&A^n\ar[r]&B^n\ar[r]&C^n\ar[r]&0}$$ in $\md\Gamma$, obtained by applying $K\otimes_k-$. It seems like this should imply that there is a short exact sequence $$\xymatrix{0\ar[r]&A\ar[r]&B\ar[r]&C\ar[r]&0}$$ in $\md\Gamma$ as well, and thus that $M^n\le{deg}N^n$ implies $M\le{deg}N$. Unfortunately this is not true in general.

The next example, which is a variant of the Carlson example mentioned in the introduction, shows that $M^n\le{deg}N^n$ does not imply $M\le{deg}N$.
\begin{ex}
Let $\Lambda$ be the exterior $k$-algebra in two variables $X$ and $Y$. Let $f\in\Lambda$ be an element of degree 1, i.e. $f=aX+bY$ for some $a,b\in k$, and let $(f)$ be the submodule of $\Lambda$ generated by $f$. There is an exact sequence of $\Lambda$-modules $$\xymatrix{0\ar[r]&(f)\ar[r]&\Lambda\ar[r]&(f)\ar[r]&0},$$ which shows that $\Lambda\le{deg}(f)^2$. If $g\in\Lambda$ is another element of degree 1, then we have $\Lambda^2\le{deg}((f)\oplus (g))^2$. However, by Theorem 5.4 in \cite{SmalVal} we have $\Lambda\le{deg}(f)\oplus (g)$ if and only if $(f)\simeq (g)$. As in the original Carlson example, we have $\Lambda\le{vdeg}(f)\oplus (g)$ for all $f,g$.
\end{ex}

Adding a suitable $K$-structure, this shows that we may have $M\le{deg}^\Lambda N$ and $M\le{vdeg}^\Gamma N$ without having $M\le{deg}^\Gamma N$.

We give one more example of this, due to S. Oppermann and S. O. Smal\o{}. Here we also see that we can have a monomorphism from $A^2$ to $B^2$ without having any monomorphisms from $A$ to $B$.

\begin{ex}\label{smalopp}
Let $\Lambda$ be the exterior $k$-algebra in three variables $X$, $Y$ and $Z$. Let $\r$ be its radical and $S$ the simple $\Lambda$-module. The $\Lambda$-homomorphism $f:(\Lambda/\r^2)^2\to (\r/\r^3)^2$ given by right multiplication with the matrix $\smatr{X&Y\\Y&Z}$ is a monomorphism, but there are no $\Lambda$-monomorphisms from $\Lambda/\r^2$ to $\r/\r^3$. Since $\coker f$ is semisimple, we have an exact sequence $$\eta:\xymatrix{0\ar[r]&(\Lambda/\r^2)^2\ar[r]&(\r/\r^3)^2\ar[r]&(S^2)^2\ar[r]&0},$$
but there is no exact sequence $$\xymatrix{0\ar[r]&\Lambda/\r^2\ar[r]&\r/\r^3\ar[r]&S^2\ar[r]&0}.$$
The exact sequence $\eta$ shows that $(\r/\r^3)^2\le{deg}(\Lambda/\r^2\oplus S^2)^2$, and thus $\r/\r^3\le{Hom}\Lambda/\r^2\oplus S^2$. 

We also have $\r/\r^3\le{vdeg}\Lambda/\r^2\oplus S^2$. There are exact sequences
$$\xymatrix{0\ar[r]&\Lambda/\r^2\ar[r]^{\smatr{X\\Z}}&\r/\r^3\oplus (Z)/\r^3\ar[r]&\r/(XZ)\ar[r]&0},$$
$$\xymatrix{0\ar[r]&(XZ)\ar[r]&\r/(XZ)\ar[r]&(XZ,YZ)\ar[r]&0},$$
$$\xymatrix{0\ar[r]&(YZ)\ar[r]&(XZ,YZ)\ar[r]&S\ar[r]&0},$$
$$\xymatrix{0\ar[r]&(Z)/\r^3\ar[r]&(XZ)\oplus(YZ)\ar[r]&S\ar[r]&0},$$
which show that $\r/\r^3\oplus (Z)/\r^3\le{deg}\r/(XZ)\oplus \Lambda/\r^2\le{deg}(XZ)\oplus (XZ,YZ)\oplus \Lambda/\r^2\le{deg} (XZ)\oplus(YZ)\oplus S\oplus \Lambda/\r^2\le{deg}\Lambda/\r^2\oplus S^2\oplus (Z)/\r^3$.

There is no degeneration though, as we will see from Proposition \ref{subdim} below.

\end{ex}

In both examples we have a virtual degeneration, so it is possible that $M^n\le{deg}N^n$ implies $M\le{vdeg}N$.
Note also that the exterior algebras do not satisfy the statements of Theorem \ref{iso-hom}. Thus it is still possible that in this more restricted case  $M^n\le{deg}N^n$ also implies $M\le{deg}N$.

To see that $\r/\r^3$ does not degenerate to $\Lambda/\r^2\oplus S^2$ in Example \ref{smalopp}, we will look at their submodules. The 4-dimensional submodule $\Lambda/\r^2\subseteq \Lambda/\r^2\oplus S^2$ is generated by one element. In $\r/\r^3$, on the other hand, any submodule generated by one element is at most 3-dimensional. This turns out to be impossible if we have a degeneration.

For a $\Lambda$-module $M$ and a natural number $i$, let $\Sub_i M$ be the set of submodules of $M$ that are generated by $i$ elements. We have a function $f_i:\md\Lambda\to\N$ for each $i$, given by $$f_i(M)=\max_{N\in\Sub_i M}\dim_k N.$$

\begin{prop}\label{subdim}
Let $k$ be an algebraically closed field and $\Lambda$ a finite-dimensional $k$-algebra. Let $M$ and $N$ be $\Lambda$-modules such that $M\le{deg}N$. Then $f_i(M)\geq f_i(N)$ for all $i$.
\end{prop}

\begin{proof}
We want to show that for any $m,d,i\in\N$, the set $\{X\in\md_d \Lambda|f_i(X)\leq m\}$ is closed in $\md_d\Lambda$. Let $\{\lambda_1,\ldots,\lambda_n\}$ be a basis for $\Lambda$. For any $i$-tuple $\mathbf{x}=(x_1,\ldots,x_i)$ of elements in $k^d$, we have a function $\phi_{\mathbf{x}}:\md_d\Lambda\to \M_{d\times ni}(k)$ given by 
$$Y \mapsto \matr{&&&&&\\Y(\lambda_1) x_1&\cdots&Y(\lambda_n)x_1&Y(\lambda_1) x_2&\cdots&Y(\lambda_n) x_i\\&&&&&}.$$
The columns of $\phi_{\mathbf{x}}(\rho)$ span the submodule of $Y$ generated by $\{x_1,\cdots, x_i\}$, thus the dimension of that submodule equals the rank of $\phi_{\mathbf{x}}(Y)$. Let $Z_m\subseteq\M_{d\times ni}(k)$ be the set of matrices with rank at most $m$. Then the set of modules where $\{x_1,\ldots,x_i\}$ generates an at most $m$-dimensional submodule is the inverse image of $Z_m$, and we have $$\{X\in\md_d \Lambda|f_i(X)\leq m\}=\bigcap_{\mathbf{x}\in(k^d)^i}\phi_\mathbf{x}^{-1}(Z_m),$$ which is closed since the maps $\phi_{\mathbf{x}}$ are continuous and $Z_m$ is closed.

Hence the closure of the isomorphism class of $M$ is contained in $\{\rho\in\md_d \Lambda|f_i(\rho)\leq f_i(M)\}$, so if $M\le{deg}N$ we have $f_i(N)\leq f_i(M)$.
\end{proof}

It follows immediately that we cannot have a degeneration in Example \ref{smalopp} if the field is algebraically closed. Even if the field is not closed, a degeneration is not possible.  If there were a degeneration, applying $\overline{k}\otimes_k-$ to its Riedtmann-sequence would show that $N=\overline{k}\otimes_k(\Lambda/\r^2\oplus S^2)$ is a degeneration of $M=\overline{k}\otimes_k \r/\r^3$. But $f_1(M)=3$ and $f_1(N)=4$, so by Proposition \ref{subdim} we have $M\not\le{deg} N$ and consequently $\r/\r^3\not\le{deg}\Lambda/\r^2\oplus S^2$.

\section{Endomorphism rings}
\label{ringsection}

Let $k$ be a field, $\Lambda$ a $k$-algebra and $M$ and $N$ $\Lambda$-modules such that $M\le{deg}N$. Since also $M\le{Hom} N$, we have ${}_\Lambda[M,M]\leq{}_\Lambda[M,N]$ and ${}_\Lambda[M,N]\leq{}_\Lambda[N,N]$, and thus  ${}_\Lambda[M,M]\leq{}_\Lambda[N,N]$. If $M\not\simeq N$, then this is a strict inequality. If $k$ is algebraically closed, this can be shown geometrically. For arbitrary fields it   can be seen from the following lemma, which is Lemma 5.3 from \cite{sverremilano}. 

\begin{lemma}\label{sverre}
Let  $M$ and $N$ be two nonisomorphic
 $\Lambda$-modules such that $M\le{\Hom}N$. Then we have ${}_\Lambda[N,M]<{}_\Lambda[N,N]$.
\end{lemma} 

It follows from Lemma \ref{sverre} that if $N$ is a proper $\Lambda$-degeneration of $M$, then $\End_\Lambda M$ must have strictly smaller dimension than $\End_\Lambda N$. Thus if $\End_\Lambda M$ is one-dimensional, $M$ cannot be a proper degeneration of anything.

If $k$ is algebraically closed, the only finite extension of $k$ is $k$ itself. In this case it is obvious that if $\End_\Lambda M$ is a field, then $M$ must be minimal in the Hom-order, and thus also in the degeneration order.
When $k$ is not algebraically closed, $\End_\Lambda M$ might be a field different from $k$. In this case, ${}_\Lambda[M,M]$ is greater than one, so it is not immediately obvious that $M$ should be minimal. However, it is.

\begin{prop}
Let $M$ be a $\Lambda$-module such that $\End_\Lambda (M)$ is a division ring. Then M is minimal in the Hom-order, and also in the degeneration order.
\end{prop}

\begin{proof} Assume there exists a module $N\not\simeq M$ such that $N\le{Hom} M$. By Lemma \ref{sverre} we have $0<{}_\Lambda[M,N]<{}_\Lambda[M,M]$. On the other hand, $\Hom_\Lambda(M,N)$ is a right $\End_\Lambda(M)$-module, which is free since $\End_\Lambda(M)$ is a division ring. Thus ${}_\Lambda[M,M]$ divides ${}_\Lambda[M,N]$, which is a contradiction. 

Hence $M$ is minimal in the Hom-order, and since the degeneration order is coarser than the Hom-order, $M$ is also minimal in the degeneration order.

\end{proof}

\end{document}